\documentclass[11pt]{article}
\usepackage[margin=1in,bottom=1in]{geometry}

\usepackage[utf8]{inputenc}
\usepackage[T1]{fontenc}

\usepackage{libertine}
\usepackage{amsthm}
\usepackage{amsmath}

\usepackage{amssymb}
\usepackage{stmaryrd}
\usepackage{mathrsfs}
\usepackage{mathtools}
\usepackage{thm-restate}
\usepackage[colorlinks = true,linkcolor = blue,urlcolor = blue, citecolor = blue]{hyperref}
\usepackage{url}
\usepackage{todonotes}
\usepackage{mathrsfs}
\usepackage{soul}
\usepackage{relsize}
\usepackage{enumerate}
\usepackage[all,defaultlines=3]{nowidow}
\usepackage[mathlines]{lineno}
\usepackage{multicol}
\usepackage{xspace}
\usepackage{lineno}

\usepackage{textpos}

\renewcommand{\subset}{\subseteq}

\renewcommand{\leq}{\leqslant}
\renewcommand{\geq}{\geqslant}
\renewcommand{\le}{\leq}
\renewcommand{\ge}{\geq}

\renewcommand{\setminus}{-}

\newcommand{\Oh}{{\cal O}}

\newcommand{\tup}{\bar}

\usepackage{mathtools}
\usepackage{xcolor}
\usepackage{tikz}
\usetikzlibrary{calc}
\usetikzlibrary{shapes,decorations}
\usetikzlibrary{hobby}
\usetikzlibrary{decorations.markings,intersections}

\pgfdeclarelayer{background}
\pgfdeclarelayer{foreground}
\pgfsetlayers{background,main,foreground}

\pgfkeys{%
	/tikz/on layer/.code={
		\pgfonlayer{#1}\begingroup
		\aftergroup\endpgfonlayer
		\aftergroup\endgroup
	},
	/tikz/node on layer/.code={
		\pgfonlayer{#1}\begingroup
		\expandafter\def\expandafter\tikz@node@finish\expandafter{\expandafter\endgroup\expandafter\endpgfonlayer\tikz@node@finish}%
	}
}

\usepackage{titling}
\usepackage{dsfont}
\usepackage{xspace}
\usepackage{stmaryrd}
\usepackage{mathrsfs}
\usepackage{hyperref}
\usepackage{cleveref}
\usepackage{comment}
\usepackage[inline]{enumitem}
\setlist{nosep}

\usepackage[subrefformat=parens]{subcaption} 
\usepackage{nicefrac}
\usepackage{marginnote}
\usepackage{macros}
\usepackage{tikzmacros}
\usepackage{textpos}

\newif\ifcomment
\commentfalse
\commenttrue

\newcommand{\type}{\mathsf{type}}
\newcommand{\ST}{$S$--$T$\xspace}
\renewcommand{\SS}{$S$--$S$\xspace}
\newcommand{\RR}{$R$--$R$\xspace}
\newcommand{\TT}{$T$--$T$\xspace}
\newcommand{\EP}{Erdős--Pósa\xspace}

\title{Half-integral Erd\H{o}s-P\'{o}sa property for non-null $S$--$T$ paths\thanks{%
  This work was initiated during the workshop STWOR 2023 in Będlewo, Poland, which was a part of STRUG: Stuctural Graph Theory Bootcamp, funded by the ``Excellence initiative – research university (2020-2026)'' of the University of Warsaw.
  Vera Chekan's research was supported by the DFG Research Training Group 2434 ``Facets of Complexity.''
  Colin Geniet's research was supported by the ANR projects DIGRAPHS (ANR-19-CE48-0013-01) and TWIN-WIDTH (ANR-21-CE48-0014-01).
  Meike Hatzel's research was supported by the Federal Ministry of Education and Research (BMBF), by a fellowship within the IFI programme of the German Academic Exchange Service (DAAD) and by the Institute for Basic Science (IBS-R029-C1).
  Micha\l{} Pilipczuk and Marek Sokołowski's research was supported by the project BOBR that has received funding from the European Research Council (ERC) under the European Union’s Horizon 2020 research and innovation programme with grant agreement No. 948057.
  Michał Seweryn's research was supported by a PDR grant from the Belgian National Fund for Scientific Research (FNRS).
  }}
\DeclareRobustCommand{\authorthing}{
Vera Chekan\thanks{Humboldt-Universität zu Berlin, Germany. \url{vera.chekan@informatik.hu-berlin.de}}
\and
Colin Geniet\thanks{Univ. Lyon, CNRS, ENS de Lyon, UCBL, LIP UMR5668, France. \url{colin.geniet@ens-lyon.fr}}
\and
Meike Hatzel\thanks{Discrete Mathematics Group, Institute for Basic Science (IBS), Daejeon, South Korea.
\url{research@meikehatzel.com}}
\and
Michał Pilipczuk\thanks{Institute of Informatics, University of Warsaw, Poland. \url{michal.pilipczuk@mimuw.edu.pl}}
\and
Marek Sokołowski\thanks{Institute of Informatics, University of Warsaw, Poland. \url{marek.sokolowski@mimuw.edu.pl}}
\and
Michał T.\ Seweryn\thanks{Computer Science Department, Université libre de Bruxelles, Brussels, Belgium. \url{michal.seweryn@ulb.be}}
\and
Marcin Witkowski\thanks{Faculty of Mathematics and Computer Science, Adam Mickiewicz University, Pozna\'n, Poland. \url{mw@amu.edu.pl}}}
\author{\authorthing}
\date{}

\begin{document}

\maketitle
\begin{textblock}{20}(-2.1,6)
   \includegraphics[width=80px]{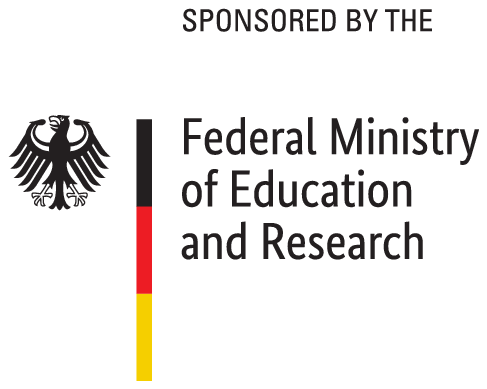}
\end{textblock}
\begin{textblock}{20}(-2.1,7.1)
 \includegraphics[width=60px]{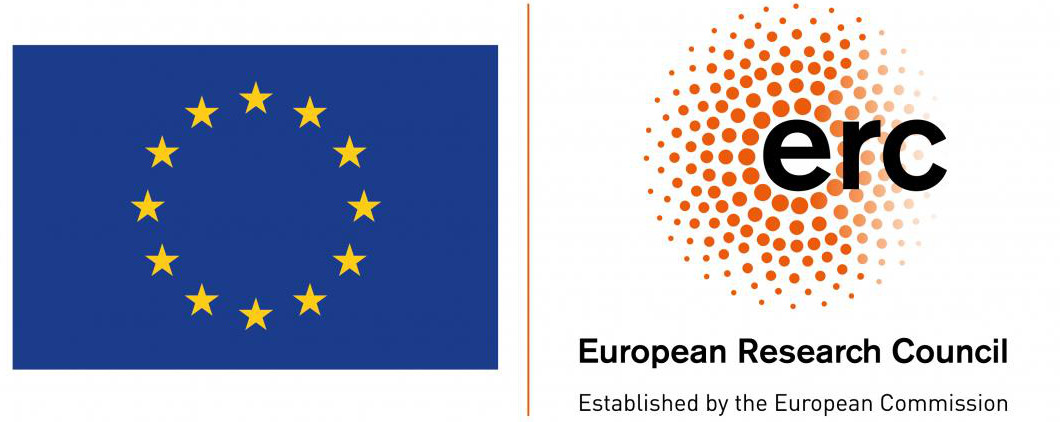}
\end{textblock}

\begin{abstract}
    For a group $\Gamma$, a {\em{$\Gamma$-labelled graph}} is an undirected graph $G$ where every orientation of an edge is assigned an element of $\Gamma$ so that opposite orientations of the same edge are assigned inverse elements. A path in $G$ is {\em{non-null}} if the product of the labels along the path is not the neutral element of $\Gamma$.

    We prove that for every finite group $\Gamma$, non-null $S$--$T$ paths in $\Gamma$-labelled graphs exhibit the half-integral Erd\H{o}s-P\'osa property.
    More precisely, there is a function $f$, depending on $\Gamma$, such that for every $\Gamma$-labelled graph $G$, subsets of vertices $S$ and $T$, and integer $k$, one of the following objects~exists:
	\begin{itemize}[nosep]
	 \item a family $\cal F$ consisting of $k$ non-null $S$--$T$ paths in $G$ such that every vertex of $G$ participates in at most two paths of $\cal F$; or
	 \item a set $X$ consisting of at most $f(k)$ vertices that meets every non-null $S$--$T$ path in $G$.
	\end{itemize}
	This in particular proves that in undirected graphs $S$--$T$ paths of odd length have the half-integral Erd\H{o}s-P\'osa~property.
\end{abstract}

\newpage

\section{Introduction}

A classic theorem of Erd\H{o}s and P\'osa~\cite{ErdosPosa65} is the following: For every undirected graph $G$ and integer~$k$, one can find in $G$ a set of $k$ vertex-disjoint cycles, or a set of $f(k)$ vertices that meets all the cycles, where $f(k)\in \Oh(k\log k)$. Since the establishment of this result in 1965, the existence of this kind of a functional relation between the (maximum) packing number and the (minimum) hitting number has been investigated for a wide range of families of objects in graphs. Families that enjoy such a relation are said to satisfy the {\em{\EP property}}. We refer to the dynamic listing maintained by Raymond~\cite{listing} for a comprehensive overview of this research area.

A natural idea in this context is to consider objects rooted at some set(s) of vertices. For instance, Menger's Theorem can be interpreted as follows: \ST paths\footnote{An {\em{\ST path}} is a path with one endpoint in $S$ and the other in $T$. The internal vertices are also allowed to belong to~$S$ or~$T$.} in graphs enjoy the \EP property with $f(k)=k-1$. Another interesting setting is when one puts restrictions on the lengths of considered paths or cycles. For example, Reed~\cite{bruce1999mangoes} proved that while odd cycles in undirected graphs do not have the \EP property, they enjoy the {\em{half-integral \EP property}}: In every undirected graph $G$ one can find a family $\cal F$ consisting of $k$ odd cycles such that every vertex participates in at most two cycles from~$\cal F$, or there is a set of at most $f(k)$ vertices that meets all odd cycles in $G$.

\newcommand{\vE}{\vec E}

In this work, we consider the natural common generalization of the two examples described above. While our original goal was to study odd \ST paths, our proofs work in the more general setting of group-labelled graphs, which we now recall. For an undirected multigraph $G$, by $V(G)$ and $E(G)$ we denote the vertex set and the edge set of $G$, respectively. Then, the set of {\em{arcs}} in $G$ is $\vE(G)=\{(u,v),(v,u)\colon uv\in E(G)\}$; that is, for every edge $e=uv$ of $G$ we include its two {\em{orientations}} $(u,v)$ and $(v,u)$. For every arc $a$, the reverse arc originating from orienting the same edge in the other direction is denoted by $a^{-1}$. Next, for a group $\Gamma$, a {\em{$\Gamma$-labelled graph}} is an undirected graph $G$ together with a labelling $\lambda_G\colon \vE(G)\to \Gamma$ satisfying $\lambda_G(a^{-1})=\lambda_G(a)^{-1}$ for every arc $a\in \vE(G)$. We omit the subscript if the graph is clear from the context. A {\em{path}} in a $\Gamma$-labelled graph $G$ is a sequence $P=(u_0,a_1,u_1,a_2,u_2,\ldots,u_{p-1},a_p,u_p)$ consisting alternately of vertices and arcs so that each arc $a_i$, $i\in \{1,\ldots,p\}$, has tail $u_{i-1}$ and head $u_i$, and the vertices $u_0,u_1,\ldots,u_p$ are pairwise different. We define the {\em{value}} of $P$ as the product of the labels along $P$:
$$\lambda(P)=\lambda(a_1)\cdot \lambda(a_2)\cdot \ldots \cdot \lambda(a_p),$$
where $\cdot$ denotes the group operation in $\Gamma$. Then $P$ is {\em{non-null}} if $\lambda(P)\neq 1_\Gamma$, where $1_\Gamma$ is the neutral element of $\Gamma$. Clearly, $P$ is non-null if and only if its reversal $P^{-1}=(u_p,a_p^{-1},u_{p-1},\ldots,u_1,a_1^{-1},u_0)$ is non-null, for $\lambda(P^{-1})=\lambda(P)^{-1}$. Hence, we can speak about non-null unoriented paths. Observe that when~$\Gamma = \ZZ / 2\ZZ$ and every arc is given value~$1$, the non-null paths are exactly paths of odd length.

It is not difficult to convince oneself that odd \ST paths in undirected graphs do not enjoy the \EP property: a counterexample is presented in \cref{fig:odd-grid}. The main result of this work is that once we relax the question by allowing half-integrality, the answer becomes positive.

\begin{restatable}{theorem}{mainTheorem}\label{thm:main}
 For every finite group $\Gamma$ there exists a function $f\colon \N\to \N$ such that the following holds. Let $G$ be a $\Gamma$-labelled graph, $S$ and $T$ be vertex subsets in $G$, and $k$ be an integer. Then $G$ contains at least one of the following objects:
 \begin{itemize}
    \item A family $\cal F$ consisting of $k$ non-null \ST paths such that every vertex of $G$ participates in at most two paths of $\cal F$.
    \item A set~$X$ consisting of at most $f(k)$ vertices such that~$X$ meets every non-null \ST path.
 \end{itemize}
\end{restatable}

In particular, \cref{thm:main} implies that odd \ST paths in undirected graphs exhibit the half-integral \EP property.

Let us briefly discuss the approach we employ in the proof of \cref{thm:main}. The idea is to first prove the result assuming the graph $G$ is highly connected, and then reduce the general case to the highly connected case. The concept of being ``highly connected'' is formalized through the notion of {\em{unbreakability}}, extensively used in parameterized algorithms; see for instance~\cite{ChitnisCHPP16,CyganLPPS19,KawarabayashiT11,LokshtanovR0Z18}\footnote{The notion of unbreakability was first defined explicitly in the work of Cygan, Lokshtanov, Pilipczuk, Pilipczuk, and Saurabh~\cite{CyganLPPS19}, but it appears implicitly also in the earlier works.}. Roughly speaking, a graph $G$ is $(q,k)$-unbreakable if for every separation of $G$ of order less than $k$, one of the sides has fewer than $q$ vertices; or contrapositively by Menger's theorem, any two sets consisting of $q$ vertices each can be connected by $k$ vertex-disjoint~paths.

We first prove (\cref{thm:unbreakable-EP} in \cref{sec:unbreakable}) that assuming $G$ is $(q,k)$-unbreakable, we can always find a congestion-$2$ packing\footnote{A {\em{congestion-$c$ packing}} of paths is a family of paths such that every vertex participates in at most $c$ of those paths.}
consisting of $k$ non-null \ST paths, or a hitting set for such paths of size at most $4q+2k-6$. For this, we use a result of
Chudnovsky, Geelen, Gerards, Goddyn, Lohman, and Seymour~\cite{chudnovsky2006packing} to either find the desired hitting set, or construct a family of $q$ disjoint non-null \SS paths. These \SS paths can be consequently linked to $T$ using $k$ disjoint paths using the assumption about unbreakability. This creates a congestion-$2$ packing of $k$ ``tripods'', where each tripod consists of a non-null \SS path $P$ that is connected to $T$ by a path $Q$ that shares only an endpoint with $P$. It is then easy to observe that each tripod contains a non-null \ST path, giving us the desired congestion-$2$ packing of $k$ non-null \ST~paths.

For the reduction to the unbreakable case, we use another technique borrowed from the area of parameterized algorithms: {\em{gadget replacement}}. Namely, suppose the given graph $G$ is not $(q,k)$-unbreakable for some large $q$ depending on $k$, and let $(A,B)$ be a witnessing separation: $A$ and $B$ are vertex subsets with $A\cup B=V(G)$, $|A|\geq q$, $|B|\geq q$, $|A\cap B|<k$, and no edges between $A\setminus B$ and $B\setminus A$. If both $G[A\setminus B]$ and $G[B\setminus A]$ contain a non-null \ST path, then the half-integral packing numbers in those subgraphs must be strictly smaller than in $G$, and we can induct: find a bounded-size hitting set $X_A$ in $G[A\setminus B]$ and a bounded-size hitting set $X_B$ in $G[B\setminus A]$, and return $X_A\cup X_B\cup (A\cap B)$ as a hitting set. Otherwise, one of those induced subgraphs, say $G[B\setminus A]$, contains no non-null \ST path. We then prove that provided $q$ is large enough depending on $k$, $G[B]$ can be replaced with a strictly smaller graph of the same ``type'', so that the graph $G'$ obtained from $G$ by applying the replacement has exactly the same congestion-$2$ packing number and the same hitting number for non-null \ST paths. This allows us to apply induction again, this time on the number of vertices of the considered graph. The precise definition of the ``type'' of $G[B]$ and the argument for the correctness of the replacement are quite technical and delicate.

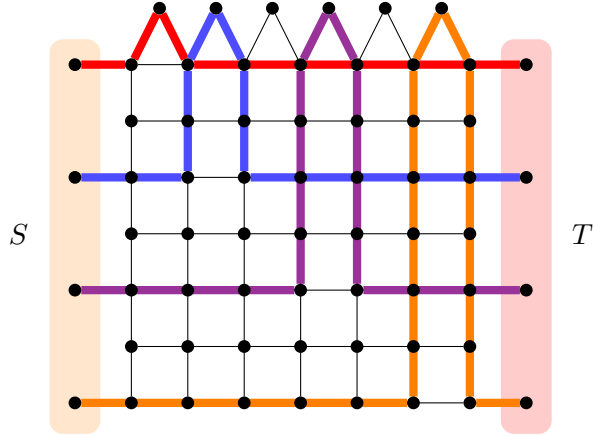
\begin{figure}[!ht]
  \begin{center}
    \begin{tikzpicture}[scale=0.75]
      \def\xx{7}
      \def\yy{7}

      \pgfdeclarelayer{fg}
      \pgfsetlayers{main,fg}

      \tikzstyle{vertex}+=[inner sep=.5mm]

      \fill[orange!20,rounded corners=5] (-0.45,0.45) rectangle (0.45,\yy+0.45);
      \fill[ red!20,rounded corners=5] (\xx+0.55,0.45) rectangle (\xx+1.45,\yy+0.45);
      \node at (-1,\yy/2+0.5) {$S$};
      \node at (\xx+2,\yy/2+0.5) {$T$};

      \begin{pgfonlayer}{fg}
      \foreach \x in {1,...,\xx}{
        \foreach \y in {1,...,\yy}{
          \node[vertex] (\x\y) at (\x,\y) {};
        }
      }
      \end{pgfonlayer}

      \foreach \x in {1,...,\xx}{
        \draw (\x1) -- (\x\yy);
      }
      \foreach \y in {1,...,\yy}{
        \draw (1\y) -- (\xx\y);
      }
      \foreach \y in {1,3,...,\yy}{
        \node[vertex] (s\y) at (0,\y) {};
        \node[vertex] (t\y) at (\xx+1,\y) {};
        \draw (s\y) -- (1\y);
        \draw (t\y) -- (\xx\y);
      }
      \foreach \x in {2,...,\xx}{
        \node[vertex] (v\x) at (\x-0.5,\yy+1) {};
        \draw (\x\yy) -- (v\x) -- (\x-1,\yy);
      }

      \foreach \py/\px/\col in {7/1/red,5/2/blue!70!white,3/4/violet!80!white,1/6/orange}{
        \pgfmathtruncatemacro{\pxx}{\px+1}
        \draw[line width=1.1mm,\col] (s\py) -- (\px\py) -- (\px\yy) -- (v\pxx) -- (\pxx\yy) -- (\pxx\py) -- (t\py);
      }
    \end{tikzpicture}
  \end{center}
  \caption{
    A construction showing that
    odd \ST paths do not exhibit the \EP property. The graph is obtained from the $(2n+1)\times (2n+1)$ grid by attaching a degree-$1$ vertex of $S$ to every second vertex of the left side, a degree-$1$ vertex of $T$ to every second vertex of the right side, and a triangle to every edge of the top side. On one hand, every vertex subset $X$ of size smaller than $n$ avoids at least one row containing vertices of $S$ and $T$, as well as two consecutive columns together with the triangle joining them; hence there is an odd \ST path not meeting $X$. On the other hand, every odd \ST path has to visit the top side of the grid, hence there are no two disjoint odd \ST paths.
    There is, however, a large congestion-$2$ packing of odd \ST paths, highlighted through colors. The example is a slightly adapted construction from the work of Bruhn, Henlein, and Joos~\cite{henning2018frames}, which in turn is inspired by the {\em{Escher wall}} of Lov\'asz and Schrijver, see the work of Reed~\cite{bruce1999mangoes}.
  }
  \label{fig:odd-grid}
\end{figure}

\paragraph*{Related work.} There is extensive literature on \EP-type statements in group-labelled graphs.
Chudnovsky, Geelen, Gerards, Goddyn, Lohman, and Seymour~\cite{chudnovsky2006packing} proved that non-null \SS paths in group-labelled graphs admit the \EP property, generalizing an earlier result on odd \SS paths~\cite{geelen2009oddminor}. Note that the example of \cref{fig:odd-grid} shows that in these results, it is important that we consider \SS paths and not \ST paths; that is, both endpoints of every path are required to belong to the same vertex subset~$S$. Huynh, Joos, and Wollan~\cite{2019unifiedEP} proved that non-null $S$-cycles\footnote{An {\em{$S$-cycle}} is a cycle that intersects $S$, and it is {\em{non-null}} in a $\Gamma$-labelled graph if the product of the labels along the cycle is not equal to $1_\Gamma$. It is easy to see that this definition does not depend on the choice of the starting point on the cycle.} in group-labelled graphs enjoy the half-integral \EP property, which generalized earlier results for $S$-cycles~\cite{KakimuraKM11,PontecorviW12} (here, even the integral \EP property holds) and for odd $S$-cycles~\cite{KakimuraK13}. In fact, Huynh et al.~proved a more general statement that involves labelling arcs with elements of the direct product of two groups, and considering a cycle to be non-null if its value has non-neutral elements on both coordinates. More corollaries can be derived from this formulation; see~\cite{2019unifiedEP} for details. A proof of the half-integral \EP property for non-null cycles in group-labelled graphs (not necessarily passing through a prescribed set of vertices) was also reported by Lokshtanov, Ramanujan, and Saurabh~\cite{LokshtanovRS17}. Finally, Gollin, Hendrey, Kawarabayashi, Kwon, and Oum~\cite{2024grouplabelled} studied half-integral \EP statements for cycles in graphs with edges labelled with multiple Abelian groups. However, they define group-labelled graphs somewhat differently: they restrict attention to Abelian groups, and every edge $e$ traversed by a path $P$ contributes with the same value $\lambda(e)$ to the value of $P$, regardless in which direction $e$ is traversed along $P$.

We remark that in all the works mentioned above, the proofs either exploit tight duality statements {\em{\`a~la}} Tutte-Berge formula or Mader's $A$-path Theorem, or follow the general approach paved by Reed in~\cite{bruce1999mangoes} that relies on tools from the theory of Graph Minors, particularly the duality between treewidth and walls. The way we induct on small-order separations with non-null \ST paths on both sides is essentially the same as in multiple previous works, see e.g.~\cite[3.3]{bruce1999mangoes}. However, the methodology of using gadget replacement to explicitly reduce the problem to the setting of unbreakable graphs seems to be new in the context of \EP-type results.




\section{Unbreakable graphs}\label{sec:unbreakable}

In this section we give a proof of \cref{thm:main} under the assumption that the considered graph is suitably unbreakable. The reduction of the general case to the unbreakable case is provided in the next sections. First, we need to introduce the concept of unbreakability formally.

Recall that a \emph{separation} in a graph~$G$ is a pair of vertex subsets~$(A,B)$ such that~$A \cup B = V(G)$ and there is no edge with one endpoint in $A \setminus B$ and second in $B \setminus A$. The \emph{order} of the separation is~$\Abs{A \cap B}$. With these notions in place, the definition of unbreakability reads as follows.

\begin{definition}
Let $q,k\in \N$.
We call a graph~$G$ {\em{$(q,k)$-unbreakable}} if for every separation~$(A,B)$ of $G$ of order less than~$k$, we have
$\Abs{A} < q$ or $\Abs{B} < q$.
\end{definition}

We remark that the definition used in~\cite{CyganLPPS19} differs by details of no consequence; we find the formulation presented above more convenient to work with.
Note that Menger's Theorem implies that if a graph $G$ is $(q,k)$-unbreakable, then for every pair of vertex subsets~$S$ and $T$, each of cardinality at least~$q$, there are~$k$ vertex-disjoint \ST paths.

%

In our reasoning we use the following result of Chudnovsky et al.~\cite{chudnovsky2006packing} about the \EP property of non-null \SS paths.

\begin{theorem}[\cite{chudnovsky2006packing}]
  \label{thm:non-null-S-paths-EP}
  Let $\Gamma$ be a group, $G$ be a $\Gamma$-labelled graph, $S$ be a set of vertices of $G$, and $k$ be an integer. Then $G$ contains at least one of the following objects:
  \begin{itemize}
    \item A family of $k$ vertex-disjoint non-null \SS paths.
    \item A set of at most $2k-2$ vertices that meets every non-null \SS path.
  \end{itemize}
\end{theorem}

Note that the bound of $2k-2$ in \cref{thm:non-null-S-paths-EP} is independent of the group~$\Gamma$. In fact, the result holds even when $\Gamma$ is infinite. This stands in contrast with \cref{thm:main}, where we can only claim a bound that is dependent on the group $\Gamma$, which is required to be finite. We elaborate more on this in~\cref{sec:conclusions}.

With all the tools in place, we can state and prove the statement for unbreakable graphs. Recall that we say that a family of paths has {\em{congestion $c$}} if every vertex belongs to at most $c$ paths from the~family.

\begin{proposition}
  \label{thm:unbreakable-EP}
  Let $\Gamma$ be a group and $G$ be a $\Gamma$-labelled graph that is $(q,k)$-unbreakable, for some $q,k\in \N$. Further, let $S$ and $T$ be subsets of vertices of $G$. Then $G$ contains at least one of the following objects:
  \begin{itemize}
    \item A family of~$k$ non-null \ST paths with congestion $2$.
    \item A set of at most~$4q+2k-6$ vertices that meets every non-null \ST path.
  \end{itemize}
\end{proposition}
\begin{proof}
  Letting $R=S\cup T$, we first apply \cref{thm:non-null-S-paths-EP} to~$R$, thus obtaining one of the following objects:
  \begin{itemize}
  \item a family $\cal F$ consisting of~$2q+k-2$ vertex-disjoint non-null \RR paths, or
  \item a vertex subset~$X$ of size at most~$4q+2k-6$ that meets every non-null \RR path.
  \end{itemize}
  In the second case, $X$ also meets every non-null \ST path, hence it satisfies the prerequisite of the second outcome of the proposition statement. Hence, we assume that there is a family $\cal F$ as in the first case.

  If~$k$ among the paths of $\cal F$ have one endpoint in~$S$ and the other endpoint in~$T$,
  then they form a family of $k$ non-null \ST paths satisfying the prerequisite of the first outcome of the proposition statement, and we are done.
  We may thus assume that $\cal F$ contains~$2q-1$ non-null paths that are \SS paths or \TT paths.
  Without loss of generality, assume that $q$ of them are \SS paths, and denote them by $P_1,\dots,P_q$.

  For each path $P_i$, pick an arbitrary vertex~$c_i$ traversed by $P_i$, and define~$C = \{c_1,\dots,c_q\}$.
  We may assume that~$\Abs{T} \ge q$, as otherwise $T$ itself is a set of size at most $q-1\leq 4q+2k-6$ (we recall that $k, q \geq 1$ holds) that meets all non-null \ST paths.
  Since $G$ is $(q,k)$-unbreakable, there are~$k$ vertex-disjoint $C$--$T$ paths, say~$Q_1,\dots,Q_k$.
  By renaming the paths~$P_1,\ldots,P_q$ if necessary, let us assume that~$Q_i$ connects~$T$ with~$c_i$.
  Furthermore, by shortening~$Q_i$ and changing~$c_i$ to be the first vertex of intersection between~$P_i$ and~$Q_i$,
  we may also assume that~$P_i$ and~$Q_i$ are disjoint except for~$c_i$.
  Then~$P_i \cup Q_i$ forms a \emph{tripod} with three disjoint paths starting from the center~$c_i$,
  one ending in~$T$ (namely~$Q_i$) and two ending in~$S$ (resulting from~$P_i$).

  \begin{claim}
    \label{clm:tripod-path}
    Each tripod~$P_i \cup Q_i$ contains a non-null \ST path.
  \end{claim}
  \begin{claimproof}
    For simplicity, call $P_i = P$, $Q_i = Q$, and $c_i = c$.
    Let~$A$ and $B$ be the two subpaths of~$P$ connecting~$c$ to~$S$, say with endpoints~$a,b \in S$, respectively.
    If we consider $P$ as oriented from~$a$ to~$b$, as well as~$A$ as oriented from~$a$ to $c$ and $B$ as oriented from $b$ to $c$, then the values of those paths satisfy
    \[\lambda(P) = \lambda(A) \cdot \lambda(B)^{-1}.\]
    Since~$P$ is a non-null path, it follows that~$\lambda(A)\neq \lambda(B)$. Considering the paths $A\cup Q$, $B\cup Q$, and $Q$ as oriented so that they start at $a$, $b$, and $c$, respectively, we have
    \[\lambda(A\cup Q)=\lambda(A)\cdot \lambda(Q)\neq \lambda(B)\cdot \lambda(Q)=\lambda(B\cup Q).\]
    It follows that at least one of the paths $A\cup Q$ and $B\cup Q$ is a non-null  \ST path.
  \end{claimproof}

  \begin{claim}
    \label{clm:tripod-packing}
    Every vertex of~$G$ belongs to at most two of the tripods~$P_1 \cup Q_1,\dots,P_k \cup Q_k$.
  \end{claim}
  \begin{claimproof}
    By construction, the paths~$P_1,\dots,P_k$ are all vertex-disjoint, and so are the paths~$Q_1,\dots,Q_k$.
    Thus, a vertex may belong to one tripod as part of~$P_i$ and to a second as part of~$Q_j$, but no more.
  \end{claimproof}
  Thus \cref{clm:tripod-path} gives~$k$ non-null \ST paths from the~$k$ tripods,
  and \cref{clm:tripod-packing} shows that this family of paths has congestion~2.
\end{proof}

\section{Types and gadget replacement}\label{sec:types}

\Cref{thm:unbreakable-EP} proves our main result in the case of unbreakable graphs.
In the general case, we reduce the problem to~\cref{thm:unbreakable-EP}
by recursively splitting the graph along separations of small order.
More precisely, given a separation~$(A,B)$ in~$G$ with both~$A$ and~$B$ large,
we replace one of the two sides by a \emph{gadget} of bounded size,
which preserves the behavior of the replaced side with regards to congestion-$2$ packings and hitting sets of non-null \ST paths.

Preserving the behavior is formalized through a notion of \emph{types}. More precisely,
given a $\Gamma$-labelled graph~$G$ with vertex subsets $S$ and $T$,
and with a tuple of \emph{interface vertices}~$\tup u=(u_1,\dots,u_r)$,
$\type(G,\tup u)$ is a piece of information that encodes all the relevant data about congestion-$2$ packings and hitting sets of paths
from the interface vertices to themselves and to~$S$ or~$T$.
It satisfies two key properties:
\begin{enumerate}
  \item \emph{Finiteness}: For a fixed interface size~$r$ and a finite group~$\Gamma$, there are only finitely many types of $\Gamma$-labelled graphs with $r$ interface vertices.
  \item \emph{Compositionality}: Roughly speaking, whenever $(A,B)$ is a separation of a $\Gamma$-labelled graph~$G$, the subgraph induced by $B$ can be replaced by another one with the same type without changing the congestion-$2$ packing number and the hitting number of non-null \ST paths in~$G$.
\end{enumerate}
Then the idea is as follows: given a separation $(A,B)$ witnessing that $G$ is not $(q,k)$-unbreakable for some large $q$, we replace the subgraph induced by $B$ with a \emph{gadget} of the same type. Compositionality ensures that this replacement is unproblematic, while due to finiteness we can use a gadget of size dependent only on $k$. So if $q$ is larger than the maximum gadget size, the replacement always leads to a strict decrease in the size of the graph.

The gadget replacement methodology described above is widespread in the area of parameterized algorithms. See for instance the work of Lokshtanov, Ramanujan, Saurabh, and Zehavi~\cite{LokshtanovR0Z18} on reducing computational problems to the setting of unbreakable graphs. In this context, it is often convenient to consider types defined through logic.
Roughly speaking, when we consider a logic $\cal L$ (typically, the First-Order logic $\mathsf{FO}$ or the Monadic Second-Order logic $\mathsf{MSO}$, or their variants), the {\em{rank-$p$ $\cal L$-type}} of a graph $G$ with interface $\tup u$ consists of all sentences of $\cal L$ of quantifier rank at most $p$ that are satisfied in the structure consisting of $G$ with the interface vertices of $\tup u$ highlighted as constants.
Types defined in this way satisfy both finiteness and compositionality. Assuming logic $\cal L$ is suitably selected and rank $p$ is chosen large enough, the rank-$p$ $\cal L$-type of $G$ with interface $\tup u$ captures all the properties of $G$ and $\tup u$ that are of relevance to the studied problem. This gives a robust notion of a type that satisfies all the desired requirements. We refer an interested reader to the introductory article of Grohe and Kreutzer~\cite{GroheK09} for a more in-depth overview to the topic. We also remark that gadget replacement using $\mathsf{MSO}$ types was already used in the context of \EP-type results by Fiorini, Joret, and Wood~\cite{FioriniJW13}, albeit not with a reduction to the unbreakable case in~mind.

Unfortunately, for technical reasons we are unable to conduct our proof using the standard notion of $\mathsf{MSO}$ types. This is mostly connected to the fact that we are concerned with quantitative properties, such as the congestion-$2$ packing number or the hitting number for certain families of paths, and expressing those quantitative properties in logic would create a circular dependency in parameters. Therefore, we resort to defining our own, purely combinatorial notion of a type that captures exactly the properties we need. The drawback of this is that we have to prove both finiteness and compositionality for the introduced notion by hand, rather than to appeal to classic results from the literature.

Before defining types and precisely stating their properties,
let us address some minor issues of conventions. First, we speak about induced subgraphs of $\Gamma$-labelled graphs. They are defined naturally: for a subset of vertices $A$ of a $\Gamma$-labelled graph $G$, the \emph{induced subgraph} $G[A]$ consists of vertices of $A$ and all edges of $G$ with both endpoints in $A$, with labels inherited from $G$.
Second, when replacing a part of a graph, we wish to allow both the removed part and the gadget replacing it to contain vertices of~$S$ and~$T$.
The meaning of~$S$ and~$T$ may then become quite unclear.
For this reason, we consider the sets of source and target vertices
to be intrinsic parts of a graph~$G$, denoted~$S(G)$ and~$T(G)$ respectively,
using~$S$ and~$T$ when there is no ambiguity.
For instance, an \ST path in~$G$ now really means an~$S(G)$--$T(G)$ path.
When considering a subgraph~$G[A]$ induced by some subset of vertices~$A$,
the subsets are also restricted to~$A$: we set $S(G[A]) = S(G) \cap A$ and $T(G[A]) = T(G)\cap A$.
Such graphs $G$ supplied with sets $S(G)$ and $T(G)$ are called {\em{\STgraphs}}.
When we say two \STgraph $G$ and $G'$ are equal, that is $G=G'$, then this also means that $S(G)=S(G')$ as well as $T(G)=T(G').$
Note that an \STgraph can be $\Gamma$-labelled as well.

\newcommand{\pack}{\mathsf{packing}}
\newcommand{\hit}{\mathsf{hitting}}

Let us now state the two desired properties of types precisely,
before constructing a definition satisfying them.
Here, for a $\Gamma$-labelled \STgraph $G$, we define
\begin{itemize}
 \item $\pack(G)$ to be the maximum size of a family of non-null \ST paths in $G$ with congestion $2$, and
 \item $\hit(G)$ to be the minimum size of a vertex subset that meets all non-null \ST paths in $G$.
\end{itemize}

\begin{restatable}[Finiteness]{lemma}{finiteness}
  \label{lem:type-finite}
  For every~$r\in \N$ and a finite group~$\Gamma$, there are finitely many distinct types
  of $\Gamma$-labelled \STgraphs with an interface of size at most~$r$.
\end{restatable}

\vspace{-0.2cm}

\begin{restatable}[Compositionality]{lemma}{compositionality}
  \label{lem:type-compositional}
  Let $\Gamma$ be a finite group.
  Let~$G$ and $G'$ be two $\Gamma$-labelled \STgraphs,
  and~$A$ be a~subset of vertices on which they coincide: $A\subseteq V(G)\cap V(G')$ and $G[A] = G'[A]$\footnote{Note that since $G[A]$ and $G'[A]$ are considered to be \STgraphs, $G[A] = G'[A]$ also implies that $S(G)\cap A = S(G')\cap A$.}.
  Further, let~$(A,B)$ and~$(A,B')$ be separations in~$G$ and $G'$, respectively,
  with $A \cap B = A \cap B' = \{u_1,\dots,u_r\}$.
  Finally, assume that
  \begin{itemize}
    \item $\type(G[B],\tup u) = \type(G'[B'],\tup u)$ where $\tup u=(u_1,\ldots,u_r)$, and
    \item neither~$G[B \setminus A]$ nor~$G'[B' \setminus A]$ contains a non-null \ST path.
  \end{itemize}
  Then we have
  \[\pack(G)=\pack(G')\qquad\textrm{and}\qquad \hit(G)=\hit(G').\]
\end{restatable}

Call a $\Gamma$-labelled \STgraph $G$ with interface $\tup u=(u_1,\ldots,u_r)$ {\em{safe}} if in $G-\{u_1,\ldots,u_r\}$ there is no non-null \ST path.
From the finiteness we immediately derive the following statement, which gives us gadget replacement for safe graphs.

\begin{corollary}
  \label{cor:type-bounded}
  For every finite group $\Gamma$ and $r\in \N$, there exists $h\in \N$ such that the following holds: For any safe $\Gamma$-labelled \STgraph~$G$ with an interface $\tup u$ of size at most~$r$, there exists a safe $\Gamma$-labelled \STgraph $\widehat{G}$ also with interface $\tup u$ such that $\type(G,\tup u)=\type(\widehat{G},\tup u)$ and $\widehat{G}$ has at most $h$ vertices.
\end{corollary}
\begin{proof}
  Call a type $\tau$ {\em{safe-realizable}} if there is a safe graph with type $\tau$. For every safe-realizable type $\tau$, let $h_\tau$ be the minimum number of vertices in a safe graph with type $\tau$. Further, let $h$ be the supremum of numbers $h_\tau$ among the safe-realizable types $\tau$ with at most $r$ interface vertices. As there are finitely many such types by \cref{lem:type-finite}, $h$ is finite. That $h$ satisfies the required property follows directly from the~construction.
\end{proof}

Note that the argument of \cref{cor:type-bounded} is not constructive, hence the obtained bound $h$ is not a priori computable from $\Gamma$ and $r$. We discuss this issue further in \cref{sec:conclusions}.

\paragraph{Definition.} Fix a finite group $\Gamma$.
We proceed to the formal definition of the type of a $\Gamma$-labelled \STgraph~$G$
with~$r$ interface vertices~$\tup u=(u_1,\dots,u_r)$.
\begin{description}
  \item[Path] A \emph{path constraint} for an oriented path~$P$ specifies
  \begin{enumerate}
    \item the value $\lambda(P)\in \Gamma$;
    \item each of the end-vertices of~$P$, as either a specified interface vertex~$u_i$, or a marker `any vertex in~$S$' or `any vertex in~$T$'; and
    \item the non-empty subset of interface vertices through which~$P$ passes.
  \end{enumerate}

    An example of a path constraint is:
    a path from~$S$ to~$u_1$ passing through~$u_2$ (but no other interface vertex than~$u_1$ and $u_2$), and with value~$1_\Gamma$.

    Note that `an \ST path disjoint from the interface' is \emph{not} a valid path constraint: it is crucial that any path considered intersects the interface.

  \item[Path system] A \emph{$k$-path system problem} consists of~$k$ path constraints~$C_1,\dots,C_k$
    and, for each pair of indices $i,j\in \{1,\ldots,k\}$ with $i\neq j$, possibly the additional requirement that the $i$th and the $j$th path need to be disjoint.
    A {\em{solution}} to this problem consists of~$k$ paths~$P_1,\dots,P_k$ in $G$ such that each path~$P_i$ satisfies the constraint~$C_i$,
    paths $P_i$ and $P_j$ are disjoint if so required, and every vertex of $G$ appears in at most two among the paths~$P_1,\dots,P_k$.

    Note that since each path must use an interface vertex and the congestion is limited to~2,
    a path system problem is only meaningful for~$k \le 2r$ paths.

  \item[Hitting set] An \emph{$\ell$-hitting set of $k$-path system problem}
    is a set of $k$-path system problems~$\PPP_1,\dots,\PPP_m$.
    A~{\em{solution}} to this hitting set problem is a set~$X \subset V(G)$ consisting of at most~$\ell$ vertices
    such that in~$G \setminus X$, none of the problems~$\PPP_i$ have a solution.

    Note that for any~$\ell \ge r$,
    the interface $\{u_1,\dots,u_r\}$ is a trivial solution to any hitting set problem,
    since the paths in solutions to problems $\PPP_i$ are required to use interface vertices.
    Thus, we only consider hitting set problems with~$\ell \leq r$ and~$k \le 2r$.

  \item[Type] Finally, we define $\type(G,\tup u)$ to be
    the set of all $\ell$-hitting set of $k$-path system problems which do have a solution in~$G$,
    for all~$\ell \leq r$ and~$k \le 2r$.
\end{description}

Let us now prove the two key properties of types.
\finiteness*
\begin{proof}
  The number of distinct path constraints is bounded by~$\alpha\coloneqq\Abs{\Gamma}\cdot(r+2)^2 \cdot 2^r$. Next,
  the number of distinct $k$-path system problems with $k \le 2r$ is bounded by $\beta\coloneqq\sum_{k=0}^{2r} \alpha^k\cdot 2^{\binom{k}{2}}\leq 2^{\Oh(r^2)}\cdot \Abs{\Gamma}^{\Oh(r)}$. Finally, the number of distinct $\ell$-hitting set of $k$-path system problems with $\ell \leq k$ and $k \le 2r$ is bounded by $(2r)^2\cdot 2^\beta$. Hence, the number of distinct types is bounded by a triple-exponential function of~$\Abs{\Gamma}$ and~$r$.
\end{proof}

{%
	\renewcommand*{\footnote}[1]{}
\compositionality*
}%
\begin{proof}
  We first focus on proving the equality $\pack(G)=\pack(G')$.
  Consider a family of~$s$ non-null \ST paths $P_1,\dots,P_s$ with congestion~2 in~$G$, for some $s\in \N$. Suppose $k$ of those paths intersect the interface $\{u_1,\ldots,u_r\}$. By renaming the paths if necessary, we may assume that the paths $P_1,\ldots,P_k$ intersect the interface, while paths $P_{k+1},\ldots,P_s$ do not. Since $G[B\setminus A]$ does not contain any non-null \ST path, all paths $P_{k+1},\ldots,P_s$ are entirely contained in $G[A\setminus B]$. Note that since the family $P_1,\ldots,P_s$ has congestion~$2$, we have $k\leq 2r$.

  We split each~$P_i$, $i\in \{1,\ldots,k\}$, into {\em{pieces}} (subpaths)~$P_{i,1},\dots,P_{i,t_i}$ contained alternately in~$G[A]$ and~$G[B]$ so that the endpoints of each piece are the interface vertices,
  except possibly for the starting vertex of~$P_{i,1}$ in~$S$ and the ending vertex of~$P_{i,t_i}$ in~$T$.
  Note that as a corner case, this decomposition may consist of a single piece~$P_i = P_{i,1}$ entirely contained in~$G[A]$ or~$G[B]$,
  but then this piece must still pass through at least one interface vertex, because we assumed that $P_i$ intersects the interface.
  Therefore, every piece~$P_{i,j}$ intersects the interface,
  and thus corresponds to some valid path constraint~$C_{i,j}$.

  Consider then the path system problem consisting of all the path constraints~$C_{i,j}$ constructed for pieces $P_{i,j}$ that are contained in $G[B]$,
  with the requirement that the paths for~$C_{i,j}$ and~$C_{i,j'}$ be disjoint for all distinct $j,j'\in \{1,\ldots,t_i\}$
  (that is, different pieces of the same path may not intersect).
  This is a valid path system problem~$\PPP$,
  and the fact that it has a solution is encoded in~$\type(G[B],\tup u)$.
  Indeed,~$\PPP$ having a solution is the same as $G$ not admitting any hitting set of size $0$ for~$\PPP$.
  Since~$G'[B']$ has the same type as $G[B]$, it also has a solution to~$\PPP$,
  consisting of paths~$P'_{i,j}$ that satisfy constraints~$C_{i,j}$, respectively.

  We now reconstruct the \ST paths~$P'_1,\ldots,P'_k$ in~$G'$ by replacing any piece~$P_{i,j}$ contained in~$G[B]$ by the piece~$P'_{i,j}$ in $G'[B']$.
  Let us check that~$P'_1,\dots,P'_k,P_{k+1},\ldots,P_s$ is a family of non-null \ST paths in $G'$ with congestion~2.
  Firstly, since the pieces~$P'_{i,j}$ are vertex-disjoint for every fixed~$i\in \{1,\ldots,k\}$
  and pass through exactly the same interface vertices as~$P_{i,j}$,
  they correctly reassemble into an \ST path~$P'_i$ without creating loops.
  Also, $P_{i,j}$ and~$P'_{i,j}$ have exactly the same values, so~$P'_i$ gets the same value as $P_i$ and hence is non-null.
  Finally, each vertex of~$G'$ is used by at most two paths:
  for vertices in~$A \setminus B$ this is by assumption on the initial path family~$P_1,\dots,P_s$;
  for vertices in~$B \setminus A$ this is an explicit requirement for solutions of path system problems;
  and for the interface vertices, this is ensured by the fact that every piece~$P'_{i,j}$
  goes through exactly the same interface vertices as~$P_{i,j}$.

  Thus, if~$G$ contains a family of~$s$ non-null \ST paths with congestion $2$, then so does~$G'$. The proof is of course symmetric, hence the same claim also holds for $G$ and $G'$ swapped. We conclude that $\pack(G)=\pack(G')$, as required.

  Let us now adapt the argument to prove that also $\hit(G)=\hit(G')$.
  Consider a set of vertices $X\subseteq V(G)$ of minimum size that meets every non-null \ST paths in~$G$. Denote~$X_B = X \setminus A$ and observe that $|X_B|\leq r$, because otherwise replacing $X_B$ with the interface $\{u_1,\ldots,u_r\}$ in $X$ would yield a set $Y$ with $|Y|<|X|$ that would still meet all non-null \ST paths in $G$, a contradiction to the minimality of $X$.

  Let~$\mathfrak{P}$ be the set of all path system problems whose solutions in~$G[B]$ must intersect~$X_B$; that is,~$\mathfrak{P}$ is the inclusion-wise maximal hitting set problem with solution~$X_B$.
  Since $|X_B|\leq r$ and~$G'[B']$ has the same type as~$G[B]$, we conclude that in $G'[B']$ there exists a solution~$X'_B$ for~$\mathfrak{P}$ with~$\Abs{X'_B} \leq \Abs{X_B}$.

  Define~$X'=(X\setminus X_B)\cup X'_B$ and note that $\Abs{X'} \leq \Abs{X}$. We claim that~$X'$ meets all non-null \ST paths in~$G'$.
  Indeed, suppose for a contradiction that~$P'$ is a non-null \ST path in~$G'$ disjoint from~$X'$.
  As previously, we split~$P'$ into pieces~$P'_1,\dots,P'_t$ alternately contained in~$G'[A]$ and~$G'[B']$, so that the endpoints of each piece are interface vertices, except possibly for the starting vertex of $P'_1$ in $S$ and the ending vertex of $P'_t$ in $T$. Consider the path system problem~$\PPP$ describing the pieces contained in~$G'[B']$ (in $\PPP$, all pieces are required to be pairwise disjoint).
  The pieces $P'_1,\ldots,P'_t$ themselves are a solution to~$\PPP$ in~$G'[B']$ disjoint from~$X'_B$,
  hence~$\PPP$ is not in~$\mathfrak{P}$.
  By the choice of~$\mathfrak{P}$, this implies that~$\PPP$ also has a solution in~$G[B]$ disjoint from~$X_B$.
  The same replacement arguments as in the first half of the proof
  transform~$P'$ into a non-null \ST path~$P$ in~$G$ disjoint from~$X$, a contradiction.

  In summary, we constructed a vertex subset $X'\subseteq V(G')$ with $|X'|\leq |X|$ that meets all non-null \ST paths in $G'$. This shows that $\hit(G')\leq \hit(G)$, and a symmetric argument proves that $\hit(G)\leq \hit(G')$ as well. Hence $\hit(G)=\hit(G')$, as required.
\end{proof}

\section{Proof of the main result}\label{sec:main}

We are ready to assemble all the prepared tools and prove our main result, which we recall for convenience.

\mainTheorem*
\begin{proof}
  Assume that~$G$ does not contain a family of~$k$ non-null \ST paths with congestion~2. Our goal is to provide an upper bound on $\hit(G)$, the minimum size of a vertex subset that meets all non-null \ST paths, expressed as a function $f(k)$. We do this by induction first on~$k$, and then on the vertex count of~$G$. 

  Let~$q=h+1$, where $h=h(k)$ is the integer provided by \cref{cor:type-bounded}
  for the group~$\Gamma$ and $r=k-1$.
  If the graph~$G$ is $(q,k)$-unbreakable, we conclude immediately by \cref{thm:unbreakable-EP}, assuming we eventually~have \[f(k)\geq 4q+2k-6=4h(k)+2k-2.\]
  Thus, we now assume that there is a separation~$(A,B)$ witnessing that $G$ is not $(q,k)$-unbreakable: ~$\Abs{A \cap B} < k$ and~$\Abs{A},\Abs{B} \ge q$.
  There are two cases to consider.

  Suppose first that there is a non-null \ST path~$P_A$ entirely contained in~$G[A \setminus B]$ as well as a non-null \ST path~$P_B$ entirely contained in~$B \setminus A$. Then~$G[B \setminus A]$ cannot contain a family of $k-1$ non-null \ST paths with congestion~2,
      because then adding~$P_A$ would yield such a family of size~$k$ in $G$, a contradiction.\footnote{%
        Actually, we can even claim a bound of $k-2$ instead of $k-1$, since congestion~2 allows using~$P_A$ twice.
      }
      Therefore, by induction, in $G[B\setminus A]$ there is vertex subset~$X_B$ of size at most $f(k-1)$ that meets all non-null \ST paths in~$G[B \setminus A]$.
      Similarly, in $G[A\setminus B]$ there is a vertex subset $X_A$ of size at most $f(k-1)$ that meets all non-null \ST paths in~$G[A \setminus B]$.
      Then, the set $X=(A \cap B) \cup X_A \cup X_B$ meets all non-null \ST paths in~$G$, and its size is bounded by $k-1+2f(k-1)$. Therefore, provided we eventually have
      \[f(k)\geq k-1+2f(k-1),\]
      we are done in this case.

  Otherwise, without loss of generality we may assume that there are no non-null \ST paths in~$G[B \setminus A]$.
  Enumerate the interface vertices of $A \cap B$ as $\tup u=(u_1,\ldots,u_r)$ with~$r < k$, and consider the type $\tau = \type(G[B],\tup u)$, where we naturally consider $G[B]$ to be a $\Gamma$-labelled \STgraph.
  Note that $G[B]$ with interface $\tup u$ is safe.
      By the choice of~$q$, there exists a $\Gamma$-labelled \STgraph $\widehat{G}$ with interface $\tup u$ that is also safe, the type of $\widehat{G}$ is also $\tau$, and $\widehat{G}$ has fewer than~$q$ vertices.
      We obtain a new graph $G'$ from $G$ by replacing~$G[B]$ with $\widehat{G}$ in the expected way. Thus $G'$ has fewer vertices than $G$, while from
      \cref{lem:type-compositional} we infer that $\pack(G')=\pack(G)$ and $\hit(G')=\hit(G)$. By induction on the vertex count of the graph, we conclude that
      \[\hit(G)=\hit(G')\leq f(\pack(G'))=f(\pack(G)),\]
      as required.

    From the discussion above we conclude that the induction argument works if we define function $f$ recursively as follows:
    \begin{align*}
     f(0) & = 0,\\
     f(k) & = \max\left(4h(k)+2k-2,k-1+2f(k-1)\right)\quad\textrm{for }k\geq 1.\qquad\qedhere
    \end{align*}
\end{proof}

\section{Conclusions}\label{sec:conclusions}

We conclude by listing a few open questions.
\begin{itemize}[nosep]
 \item Due to the non-constructive argument used in \cref{cor:type-bounded}, our proof of \cref{thm:main} does not yield any explicit bound on the obtained function $f$. In fact, it is even unclear whether $f$ is computable. Possibly, a more direct approach could provide explicit bounds on the asymptotics of $f$.
 \item In \cref{thm:main}, the obtained function $f$ depends on the group $\Gamma$, and in particular we need to assume that $\Gamma$ is finite. Many other results on \EP properties in group-labelled graphs, particularly \cref{thm:non-null-S-paths-EP}, avoid this dependence, and provide a bound on $f$ that is independent of $\Gamma$. It would be interesting to investigate whether this dependence can be also avoided in the context of~\cref{thm:main}.
 \item The main question considered in this work can be also asked in the setting of directed graphs. Here, we consider directed $S$-to-$T$ paths and the value of a path is the product of the labels along it. Thus, every edge can be traversed in only one direction. We conjecture that in this setting, the answer is strongly negative, even for odd paths: for every $c\in \N$, odd $S$-to-$T$ paths in directed graphs do not enjoy the $\frac{1}{c}$-integral Erd\H{o}s-P\'osa property (that is, in packings we allow congestion $c$). However, we were unable to construct an example confirming this conjecture even in the setting where we may replace $\mathbb{Z}/2\mathbb{Z}$ with any finite group.
\end{itemize}

\bibliographystyle{plain}
\bibliography{biblio}
\end{document}